\documentclass[12pt]{amsart}
\usepackage[dvips]{graphicx}
\usepackage{verbatim}
\usepackage{geometry}
\usepackage{float}
\usepackage[all,2cell]{xy} \UseAllTwocells \SilentMatrices
\newtheorem{theorem}{Theorem}[section]
\newtheorem{lemma}[theorem]{Lemma}
\newtheorem{prop}[theorem]{Proposition}

\newtheorem{cor}[theorem]{Corollary}
\theoremstyle{definition}
\newtheorem{definition}[theorem]{Definition}

\numberwithin{equation}{section}
\numberwithin{figure}{section}
\numberwithin{table}{section}

\newcommand{\norm}[1]{||#1||}

\def\NN{\mathbb{N}}

\def\RR{\mathbb{R}}

\newcommand{\ran}{{\rm Ran}}

\begin{document}


\title{Dual Bases Functions in Subspaces}

\author{Scott N. Kersey}

\curraddr{Department of Mathematical Science, Georgia Southern University, Statesboro, GA 30460-8093}
\email{skersey@georgiasouthern.edu}
\thanks{}
\date{\today}
\dedicatory{ }

\begin{abstract}
In this paper we study dual bases functions in subspaces.
These are bases which are dual to functionals on larger linear space.
Our goal is construct and derive properties of certain bases obtained from
the construction, with primary focus on polynomial spaces in B-form.
When they exist, our bases are always affine (not convex), 
and we define a symmetric configuration that
converges to Lagrange polynomial bases.
Because of affineness of our bases, we are able to derive certain 
approximation theoretic results 
involving quasi-interpolation and a Bernstein-type operator.

In a broad sense, it is the aim of this paper to present a new way to view 
approximation problems in subspaces.
In subsequent work, we will apply our results to dual bases in subspaces of spline 
and multivariate polynomial spaces,
and apply this to the construction of blended function approximants used
for approximation in the sum of certain tensor product spaces.
\end{abstract}

\maketitle

\vskip 10pt

\section{Introduction}

Let $X$ be a finite dimensional vector spaces (of dimension $n$).
Of fundamental importance is the basis.
There are various reasons to choose a particular basis, and each
basis has advantages and disadvantages.
Often it is the action of certain functionals that lends importance to a particular basis.
For example, the Lagrange basis is important because it is dual to point evaluation.
In general, any basis $\Phi^n = [\Phi_1, \ldots, \Phi_n]$ for $X$ has a dual map
$\Lambda^n = [\lambda_1, \ldots, \lambda_n]$ for $X^*$ satisfying $\lambda_i \Phi_j = \delta_{ij}$.
That is, $\Lambda^{nT}\Phi^n = I$, the identity matrix.
Conversely, any basis $\Lambda^n \subset X^*$ is dual to some basis $\Phi^n$ of $X$.
It is the dual map $\Lambda^n$ that extracts information from the functions that is of 
particular interest here, which arguably plays a more prominent role than the basis $\Phi^n$ itself.
This process allows one to consider the \emph{information} 
we are trying to capture as our primary goal.

The above concepts are well-known and well-studied. 
In this paper, we are interested in investigating bases for \emph{subspaces} $Y$ of $X$
that are dual to \emph{subsets} of the functionals in the map $\Lambda^n$. 
For example, suppose that $\Lambda^n = [\lambda_1, \ldots, \lambda_n]$ are linearly independent
on the n-dimensional space $X$, and $Y$ is an m-dimensional subspace of $X$ with $m<n$.
Then, 
given an injective map $s : [1:m] \rightarrow [1:n]$,
our questions are:
\begin{enumerate}
\item Is the subset $\Lambda_n(s) = [\lambda_{s(1)}, \ldots, \lambda_{s(m)}]$ of $\Lambda_n$ 
linearly independent on $Y$?
\item If linearly independent, what is the basis $D^m$ for $Y$ that is dual to $\Lambda^n(s)$
in the sense that $\Lambda^{n}(s)^TD^m = I$.
\item What are the properties of this dual basis.
\end{enumerate}

In this paper we consider the action of subsets of functionals for certain basis
on \emph{subspaces} of the original space.
This allows us to view how the subspace looks according to information on the whole space.
In some sense, we are addressing the question 
``how do you approximate with less (or not enough) information?'' 
But it is not always possible to construct such bases, 
because the functionals are not always \emph{linearly independent} on the subspace.
For example, the only subset of $\Lambda = [\delta_0, \delta_0D, \delta_0 D^2]$ that is linearly 
independent on $\ran[1,(\cdot)]$ is $\Lambda = [\delta_0, \delta_0D]$, 
not $\Lambda = [\delta_0, \delta_0D^2]$ or $\Lambda = [\delta_0D, \delta_0D]$.
As we will see in this paper, the situation is different in the Bernstein basis.

It is the objective of this paper to determine if dual bases exist for certain 
subspaces of polynomial spaces, 
and if so to compute and characterize these bases, and then determine their 
properties.
In particular, we show that in the Bernstein basis for $\$_n$,
any $m+1$-selection of the dual functionals are linearly independent on $\$_m$.
Moreover, by choosing a particular symmetric choice of the functionals, 
we show that the corresponding dual basis converge to the Lagrange polynomial basis.
Later, we derive certain approximation results concerning quasi-interpolation
and a Bernstein-type operator.

\section{Dual Bases in Subspaces $Y$ of $X$}

As defined above, $X$ is a vector space of dimension $n$ with basis $\Phi^n$
and dual map $\Lambda^n$,
and $Y$ a subspace of $X$ of dimension $m$, $m < n$, with basis $\Phi^m$.
For $y \in Y \subset X$,
we have $y = \Phi^m\alpha = \Phi^n\beta$ for some coefficient sequences 
$\alpha \in \RR^m$ and $\beta \in \RR^n$.
Applying the dual basis, we have
$$
\Lambda^{nT} \Phi^m \alpha = \Lambda^{nT}\Phi^n \beta = \beta.
$$
Hence, $\beta = E \alpha$ with $E := \Lambda^{nT} \Phi^m$.
The matrix $E$ embeds the coefficients $\alpha$ of $y$ in the basis for $Y$
to it's coefficients in the basis for $X$.
By inclusion of $Y$ into $X$, we can construct
the embedding $e := \Phi^n \circ E \circ (\Phi^m)^{-1}$
of $Y$ into $X$.
And since $\Phi^m \alpha = \phi^n E\alpha$ for all $\alpha\in\RR^m$,
it follows that $\Phi^m = \Phi^n E$.
All said, this can be visualized as in the following commutative diagram.

\begin{figure}[H]
\[
\xymatrix{ 
Y  \ar@{^{(}->}[rr]^e  && X \\
\RR^{m} \ar[u]^{\Phi^m}  \ar[rr]^E && \RR^n \ar[u]_{\Phi^n} 
}
\]
\caption{Embedding of $\$_{m,y}$ into $\$_{n,x}$.}
\end{figure}

Moreover, we note that
$$
\Lambda^{nT} \Phi^m = \Lambda^{nT} \Phi^nE = E = E \Lambda^{mT} \Phi^m 
= (\Lambda^{m} E^T)^T \Phi^m.
$$
Therefore, $\Lambda^n = \Lambda^m E^T$.
Hence, we have the following result concerning the action of $\Lambda^n$ on a subspace $Y$
of $X$.

\begin{prop}
\label{prop2}
Let $(Y,\Phi^m,\Lambda^m)$ be an $m$-dimensional subspace of a vector space $(X,\Phi^n,\Lambda^n)$,
with basis $\Phi^m$ dual to $\Lambda^m$ and $\Phi^n$ dual to $\Lambda^n$.
Then, $Y$ embeds into $X$ by the map $e = \Phi^n \circ E \circ (\Phi^m)^{-1}$
with $E = \Lambda^{nT}\Phi^m$,
and the bases transform as
$$\boxed{\Phi^m = \Phi^n E}.$$
Further,
the dual basis for $X^*$ map to the dual basis for $Y^*$ by the map
$$\boxed{\Lambda^{n} = \Lambda^m E^T}.$$
\end{prop}

Since $\Lambda^n \subset X^* \subset Y^*$, the functionals $\lambda_i^n \in X^*$ are also
functionals on $Y$.
But since $\dim Y < \#\Lambda^n$, the functionals in $\Lambda^n$ cannot be linearly independent on $Y$.
However, 
since 
$$
\Lambda^{nT} \Phi^m = \Lambda^{nT} \Phi^nE = E 
$$
is a 1-1 matrix, $\Lambda^n$ does span $Y^*$
(moreover, $\Lambda^n E$ is a basis for $Y^*$ since $(\Lambda^{n}E)^T \Phi^m = E^TE$ is invertible).
Therefore, we can trim $\Lambda^n$ to a basis for $Y^*$ by removing 
\emph{inessential} vectors (this is a standard construction, c.f. \cite{KK08}).
This leaves an $m$-subvector $\Lambda^n(s)$ of $\Lambda^n$ that is linearly independent on $Y$, 
where $s$ is an injective map $s : [1:m] \rightarrow [1:n]$
(which we call a \emph{selection map}).
Then, with $I$ the $n \times n$ identity matrix, we have
$$
\Lambda^n(s)^T \Phi^m  = \Lambda^n(s)^T \Phi^n E = (\Lambda^n I(:,s))^T \Phi^n E
= I(s,:)\Lambda^{nT} \Phi^n E = I(s,:) E = E(s,:).
$$
Therefore, an $m$-subvector $\Lambda^n(s)$ of $\Lambda^n$ is linearly independent on $Y$ 
iff $E(s,:)$ is invertible.
Moreover, when $E(s,:)$ is invertible,
the basis of $Y = \ran(\Phi^m)$ that is dual to $\Lambda^n(s)$ is $\Phi^m E(s,:)^{-1}$,
as follows by computing the change of basis $\Phi^m \rightarrow \Phi^m A$:
$$
I = \Lambda^n(s)^T \Phi^m A = E(s,:) A.
$$
Hence, the basis of $Y$ that is dual to $\Lambda^n(s)$ is $D^m := \Phi^m E(s,:)^{-1}$.
The following summarizes the above statements:

\begin{prop}
\label{prop1}
Let $(X,\Phi^n,\Lambda^n)$ be a vector space with basis $\Phi^n$
dual to $\Lambda^n$,
and let $(Y,\Phi^m,\Lambda^m)$ be a subspace with basis $\Phi^m = \Phi^n E$
dual to $\Lambda^m$.
Let $s : [1:m] \rightarrow [1:n]$ denote a selection (injective) map.
\begin{enumerate}
\item $\Lambda^nE$ is linearly independent on $Y$ and of length $m$ (hence a basis for $Y^*$).
\item There exist (at least one) selections $\Lambda^n(s)$ of $\Lambda^n$ linearly independent on $Y$.
Sometimes only one selection is linearly independent, and sometimes all
selections are linearly independent 
(this is of particular interest later in this paper).
\item $E(s,:) = \Lambda^n(s)^T \Phi^m$
\item $\Lambda^n(s)$ is linearly independent on $Y$ iff $E(s,:)$ is invertible.
\item  
If $\Lambda^n(s)$ is linearly independent on $Y$, then
the basis of $Y$ that is dual to $\Lambda^n(s)$ in the sense that
$\Lambda^{n}(s)^T D^m = I$ is $D^m = \Phi^m E(s,:)^{-1}$.
\end{enumerate}
\end{prop}

\section{Additional Properties of Dual Bases in Subspaces}

In this section we state definitions and properties useful
for certain constructions that we consider in the subsequent sections.
Firstly, by Proposition \ref{prop1}, there always exists some $s$
such that $\Lambda^n(s)$ is linearly independent on $Y$, 
and in this case we have the basis
$D^m = \Phi^m E(s,:)^{-1}$ with $E = \Lambda^n(s)^T \Phi^m$.
However, this may or may not be true for all $s$.
Hence, we give the following definition to make this distinction.

\begin{definition}
We say the embedding of $(Y,\Phi^m,\Lambda^m)$ into $(X,\Phi^n,\Lambda^n)$ 
is \emph{complete} if $\Lambda^{n}(s)$ is
linearly independent on $Y$ for all injective maps (selections) 
$s : [1:m] \rightarrow [1:n]$.
\end{definition}

\begin{prop}
Let $Y = \$^m$ be the space of polynomials of degree at most $m$,
and let $X = \$^n$ be the space of polynomials of degree at most $n$.
\begin{enumerate}
\item
Let $\Phi^m$ and $\Phi^n$ be power basis for $\$^m$ and $\$^n$.
Then, the embedding is not complete.
Moreover, $\Lambda^n(s)$ is linearly independent on $\Phi^m$ iff $s = [0:m]$.
\item
Let $\Phi^m$ and $\Phi^n$ be Bernstein bases for $\$^m$ and $\$^n$, respectively.
Then, the embedding is complete.
\end{enumerate}
\end{prop}

\begin{proof}
Part (2) is non-trivial, and  will be proved later in this paper.
For part (1), 
recall that the power basis $P^n = [1, (\cdot), \ldots, (\cdot)^n]$
dual to $\Lambda^n = [\delta_0, \delta_0D, \ldots , \delta_0D^n/n!]$,
and $(\$_m, P^m, \Lambda^m)$ is the subspace with $m < n$.
Then $P^m = P^n E$ for $E = I(:,0:m)$, 
with $I$ is the $(n+1) \times (n+1)$ identity matrix.
It follows by Proposition \ref{prop2} that
$\Lambda^n  = \Lambda^m E^T = \Lambda^m I(0:m,:)$
on $\$_m$.
Let $s$ be a selection map.
Since $I(0:m,s)$ is invertible iff $s = 0:m$, it follows that
$\Lambda^n(s)$ is linearly independent on $\$_m$ iff $s = 0:m$.
That is, on $\$_m$, $[\lambda_0^n, \ldots, \lambda_m^n] = [\lambda_0^m, \ldots, \lambda_m^m]$.
Hence, there is only one choice for the selection $s$ of $\Lambda^n$ to consider,
and this choice is $\Lambda^m = \Lambda^n(s)$.
Hence, the Lagrange basis is not complete.
\end{proof}

The next properties are characteristic of Bernstein bases,
which we consider later in this paper.

\begin{definition}
\
\begin{enumerate}
\item
Let $\Phi^n$ be a basis of finite-dimensional function space.
We say $\Phi^n$ is \emph{affine on $S$} if 
$\Phi^n(t) \alpha$ is an affine combination of $\alpha$
for all $t \in S$.
That is, $\sum_i \Phi^n_i(t) = 1$ for all $t\in S$.
If it holds for all $t\in\text{dom}(\Phi^n)$, then we say $\Phi^n$ is \emph{affine}.
\item
Let $E$ be a (real) matrix.
We say $E$ is \emph{row affine} if $\sum_j E(i,j)=1$ for all $i$,
and \emph{column affine} if $\sum_i E(i,j) = 1$ for all $j$.
\end{enumerate}
\end{definition}


\begin{lemma}
\
\label{lem1}
\begin{enumerate}
\item 
Matrix inversion preserves the property row-affine.
That is, the inverse of an invertible row-affine matrix is row-affine.
\item 
Matrix multiplication preserves the property row-affine.
That is, the product of two row-affine matrices is row-affine.
\end{enumerate}
\end{lemma}

\begin{proof}
\label{laffine}
For the first result, let $A$ be an $n\times n$ invertible row-affine matrix.
Then $\sum_{j=1}^n A(i,j) = 1$ for $i=1:n$.
Since $A^{-1}A = I$,
$$
1 = \sum_{j=1}^n I(k,j)
= \sum_{j=1}^n \sum_{i=1}^n A^{-1}(k,i) A(i,j)
= \sum_{i=1}^n A^{-1}(k,i) \sum_{j=1}^n A(i,j)
= \sum_{i=1}^n A^{-1}(k,i),
$$
for $k=1:n$.
Hence, $A^{-1}$ is row-affine.

For the second result, assume $C = AB$ with $A$ and $B$ row-affine of dimensions $m \times n$
and $n \times p$, respectively.
Then, $C$ is of dimension $m \times p$, and 
$$
\sum_{j=1}^p C(k,j)
 = \sum_{j=1}^p \sum_{i=1}^n A(k,i)B(i,j)
 = \sum_{i=1}^n A(k,i) \sum_{j=1}^p B(i,j)
 = \sum_{i=1}^n A(k,i) 
 = 1,
$$
for $k=1:m$.
Therefore, $C$ is row-affine.
\end{proof}

\begin{theorem}
\label{thm1}
Let $\Phi^m$ and $\Phi^n$ be affine bases of $Y$ and $X$, respectively, with $Y$ a subspace of $X$.
\begin{enumerate}
\item
If $\Phi^m = \Phi^nE$ for some matrix $E$, then $E$ is row affine.
\item
If, moreover, $E(s,:)$ is invertible for some selection $s : [1:m] \rightarrow [1:n]$,
then the dual basis $D^m = \Phi^mE(s,:)^{-1}$ exists and is affine.
\end{enumerate}
\end{theorem}

\begin{proof}
For (1), it follows by affineness of the two bases that
$$
1 = \sum_i \Phi^m_i(t) = \sum_i \sum_j \Phi_j^n(t) E(j,i)
= \sum_j \Phi_j(t) \sum_i E(j,i).
$$
Let $\lambda_k^n$ be the functional on $X$ such that $\lambda_k \Phi^n_j = \delta_{ij}$.
Then,
$$
\lambda_k^n(1)  = \lambda_k^n(\sum_j \Phi_j^n) 
= \sum_j \lambda_k^n(\Phi_j^n)  
= \lambda_k^n(\Phi_k^n)  = 1,
$$
and so
$$
1 = \lambda_k^n(1)
= \lambda_k^n(\sum_j \Phi_j(t) \sum_i E(j,i))
= \sum_i E(k,i).
$$
This establishes (1).

For (2), we recall that
$D^m = \Phi^mE(s,:)^{-1}$ is the basis for $Y$ dual to $\Lambda^n(s)$
with $E = \Lambda^{nT} \Phi^m$, and hence $E(s,:) = \Lambda^n(s)^T \Phi^m$.
By part (1) of this theorem, $E(s,:)$ is row affine, and by 
Lemma \ref{laffine}, $A := (E(s,:))^{-1}$ is row affine as well.
Therefore, 
\begin{align*}
 \sum_{i=0}^n D^m_i 
 &= \sum_{i=0}^m \sum_{j=0:m} \Phi_j^m A(j,i) 
 = \sum_{j=0}^m B_j^m \sum_{i=0:m} A(j,i) 
 = \sum_{j=0}^m B_j^m  = 1.
\end{align*}
And so $D^m$ is an affine basis:
\end{proof}

These results regarding affine bases will concern the Bernstein basis, which we investigate
in the remaining sections, and in particular we show that dual bases are affine.
We remark here that the same is not true of convexity.
That is, the dual bases constructed are not convex.
Indeed, matrix inversion does not preserve convexity, as it does affineness.
Morever, as it turns out, certain approximation properties do not actually require convexity.
I.e., affineness is enough. Hence, our dual bases $D^m$ will enjoy many of the same
properties as $B^m$ do, in the Bernstein-basis setting, just not convexity.


In our construction we derive dual bases in terms of data maps
that are dual on a large space.
However, it is possible to have different data maps that are both dual to the same
bases.
For example, in terms of the point-evaluation function $\delta_xf = f(x)$
and derivative operator $D$,
both maps $\Lambda = [\delta_0, \ \delta_0D]$ and 
$\tilde\Lambda = [\delta_0, \ \delta_1-\delta_0]$
are dual to the power basis $\Phi = [1, (\cdot)]$,
as can be seen by $\Lambda^T\Phi = \tilde\Lambda^T\Phi = I$.
However, even with different data maps, dual bases are the same, as shown next.
This idea is important in constructing approximation operators where some
data maps may apply and not others.
More to the point, we will derive dual basis in this paper with respect to 
dual data maps that involve differentiation, hence do not apply on $C([0,1])$.
In the last section of this paper, we define new data maps that do not
involve differentiation to derive certain properties of approximation operators
on $C([0,1])$.

\begin{prop}
\label{prop3}
Let $\Lambda^n$ and $\tilde \Lambda^n$ be data maps on $X$ both dual to the basis $\Phi^n$.
Then both maps are equivalent on $X$, and the dual bases in a subspace $Y$ are invariant
of which dual map is used.
That is, if $D^m$ and $\tilde D^m$ are bases for the subspace $Y$ that are dual to $\Phi^m$
with respect to the selection $s$ and data maps $\Lambda^n$ and $\tilde \Lambda^n$, 
respectively.
Then, $D^m = \tilde D^m$.
\end{prop}

\begin{proof}
To show that $\Lambda^n$ and $\tilde \Lambda^n$ are equivalent on $X$,
let $f = \Phi^n\alpha \in X$.
Then, since both maps are dual to $\Phi^n$, we have
$\Lambda^{nT} \Phi^n\alpha = \alpha$ and $\tilde \Lambda^{nT} \Phi^n\alpha = \alpha$.
Hence, they are equivalent on $X$.

The basis $D^m$ is dual to $\Lambda^{n}(s)$ in the sense that
$\Lambda^{nT}D^m = I$.
With this basis represented $D^m = \Phi^mA = \Phi^nEA$, for some embedding,
we get 
$$\Lambda^{nT}D^m = \Lambda^{nT} \Phi^nEA = EA.$$
Therefore,
$ I = \Lambda^{nT}(s) D^m = E(s,:)A,$
and so $A = E(s,:)^{-1}$,
which is depends only on the embedding $E$ and not the data map $\Lambda^n$.
Hence, if $\tilde A$ is the transformation matrix for $\tilde\Lambda^n$,
then $\tilde A = A$,
and so $D^m = \tilde D^m$.
\end{proof}

\section{Dual Bernstein Bases in the subspace $\$_m$ of $\$_n$}

Let $\$_n$ be the space of polynomials of degree at most $n$
and $\$_m$ the subspace of polynomials of degree at most $m$, for $m < n$.
In this section we derive bases for $\$_m$ that are dual to subsets of
the dual Bernstein functionals on $\$_n$.
We begin with some basic formulas involving Bernstein polynomials that will be used.
\begin{itemize}
\item
Bernstein basis for $\$_n$:
$B^n = [B_0^n, \ldots, B_n^n]$ with $B_i^n= \binom{n}{i}(1-\cdot)^{n-i} (\cdot)^i$,
$\sum_{i=0}^n B_i^n = 1$ and $B_i^n \geq 0$ on $[0,1]$.
That is, it forms a partition of unity on $[0,1]$. 
In particular, $B^n\alpha$ is a convex combination of $\alpha$ for all $t\in [0,1]$.

\item
Degree Elevation of Bernstein Basis (See \cite{F88, PBP2002}): 
$B^m = B^n E $ with $E$ the $(n+1) \times (m+1)$ matrix with entries
$$
E(i,j) = 
\dfrac{\binom{n-i}{m-j} \binom{i}{j}}{\binom{n}{m}} 
= \dfrac{\binom{n-m}{i-j} \binom{m}{j}}{\binom{n}{i}}
$$
for $0 \leq i \leq n$ and $0 \leq j \leq m$, with $E(i,j)=0$ if $i<j$ or $n-i<m-j$.
That is,
$$
B_j^m = \sum_{i=0}^n E(i,j) B_i^n
= \sum_{i=0}^n \dfrac{\binom{n-i}{m-j} \binom{i}{j}}{\binom{n}{m}} B_i^n.
$$

\item Dual Bernstein Functionals:
$\Lambda^n = [\lambda_0^n, \ldots, \lambda_n^n]$ with 
$$\lambda_{k}^n  = \sum_{j=0}^{k} \frac{\binom{k}{j}}{\binom{n}{j}} \frac{1}{j!} \delta_0 D^j
= \sum_{j=0}^{n-k} (-1)^j 
\dfrac{\binom{n-k}{j} }{\binom{n}{j}} \frac{1}{j!} \delta_1 D^j. $$
Therefore, $\lambda_k^nB_i^n = \delta_{ki}$ and $\Lambda^{nT}B^n = I$.
(Note that the two forms are equivalent on $\$_n$, but not on all spaces.)

\item Reduction of Dual Bernstein Basis: 
On $\$_m$, $\Lambda^{n} = \Lambda^m E^T$ (by Proposition \ref{prop2}).
Hence,
$$\lambda^n_k = \sum_{j=0}^m   E^T(j,k) \lambda^m_j
 = \sum_{j=0}^m  \dfrac{\binom{n-k}{m-j} \binom{k}{j}}{\binom{n}{m}} \lambda^m_j.$$

\end{itemize}

Hence, in the context of this paper, we have the following:


\begin{theorem}
Let $B^m$ be the Bernstein basis for $\$^m$,
and let $B^n$ be the Bernstein basis for $\$^n$ with dual map $\Lambda^n$ given above,
$m\leq n$.
Let $E$ be the degree elevation matrix.
Then,
\begin{enumerate}
\item The embedding $e$ of $(\$^m,B^m,\Lambda^m)$ into $(\$^n,B^n,\Lambda^n)$ is complete.
\item The basis for $\$^m$ dual to $\Lambda^n(s)$ can be represented $D^m = B^m E(s,:)^{-1}$,
with $D^m=B^m$ when $m=n$.
\end{enumerate}
\end{theorem}

\begin{proof}
For (1), we need to show that $\Lambda^n(s)$ is linearly independent on $\$^m$ for any 
selection $s : [0:m] \rightarrow [0:n]$.
The conversion from between the power basis $P^n$ and Bernstein basis $B^n$ for $\$_n$
can be expressed
$$
P^nD_n = B^nT_n
$$
with
$D_n := \text{Diag}([\binom{n}{j} : j=0:n])$
and
$T_n$ Pascal's (lower triangular) matrix 
$$
T^n := [\binom{i}{j} : 0 \leq i,j \leq n]
=
\begin{bmatrix}
1 & 0 & 0 & \cdots & 0 \\
1 & 1 & 0 & \cdots & 0 \\
1 & 2 & 1 & \cdots & 0 \\
\vdots  & \vdots  & \vdots  & \ddots  & \vdots \\
\binom{n}{0} & \binom{n}{1} & \binom{n}{2} &  \cdots & \binom{n}{n} 
\end{bmatrix}.
$$
This follows directly from the identity
$$
\binom{n}{j} t^j = \sum_{i=0}^n \binom{i}{j} B_i^n
$$
(see \cite{PBP2002}, section 2.8, for a short proof).
Let $d^n := (D_n)^{-1} = [1/\binom{n}{j} : j=0:n]$.
Then, 
\begin{align*}
\Lambda^n(s)^T P^m 
&= (\Lambda^n I(:,s))^T P^n I(:,0:m)  \\
&= I(s,:) \Lambda^{nT} B^nT^nd^n I(:,0:m) \\
&= I(s,:) T^nd^n I(:,0:m) \\
&= T^n(s,0:m) d^n(0:m,0:m)
\end{align*}
In \cite{K12} it was shown that the truncated Pascal matrix $T^n(s,0:m)$ is invertible 
for any $m+1$ selection $s : [0:m] \rightarrow [0:n]$ of the rows of $T^n$.
Therefore, $\Lambda^n(s)^T P^m$ is invertible,
and so $\Lambda^n(s)$ is linearly independent on $\$_m$.

For (2), we have $B^m = B^nE$ with $E$ the degree elevation matrix.
Therefore, we are exactly in the framework of Proposition \ref{prop1},
with $\Phi^m = B^m$ and $\Phi^n=B^n$.
Since the embedding is complete, we have that $E(s,:)$ is invertible for any selection map $s$,
Therefore, the dual basis  of $\$^m$ dual to $\Lambda^n(s)$ exists for any selection map $s$,
and can bre represented
$D^m = D^m = B^m E(s,:)^{-1}$.
In the case $m=n$, there is only one selection $s = [1:n]$, 
and so $\Lambda^n(s) = \Lambda^n$.
Hence, $E(s,:)=E$ is the identity matrix, and $D^m = B^m = B^n$ are therefore dual to $\Lambda^n$.
\end{proof}

Since, as is well known, $B^m$ and $B^n$ are affine bases,
the next result follows directly from Theorem \ref{thm1}.
Instead, we give a more direct proof below using properties of the degree elevation matrix $E$.

\begin{theorem}
\label{thm2}
For any selection (injective) map $s : [0:m] \rightarrow [0:n]$,
$D^m := B^mA$ is an affine basis of $\$_m$ dual to $\Lambda^n(s)$,
with $A := E(s,:)^{-1}$.
\end{theorem}

\begin{proof}
We note that
by the previous lemma $\Lambda^n(s)$ is linearly independent on $\$_m$.
Therefore, $\Lambda^n(s)^T B^m$ is invertible.
By Proposition \ref{prop1}, $\Lambda^n(s)^T B^m = E(s,:)$
and $D^m = B^m E(s,:)^{-1}$ is the basis for $\$_m$ dual to $\Lambda^n(s)$.
Recall that
$$
E(i,j) = \dfrac{\binom{n-i}{m-j} \binom{i}{j}}{\binom{n}{m}}.
$$
By the Chu-Vandermonde identity 
\begin{align*}
 \binom{n}{m} \sum_{j=0}^m E(i,j) = 
 \sum_{j=0}^m \binom{n-i}{m-j} \binom ij
 &= \sum_{j=0}^m \binom{n-i}{j} \binom{i}{m-j}
 = \binom{n}{m},
\end{align*}
and so $E(s,:)$ is row affine.
By Lemma \ref{laffine}, $(E(s,:))^{-1}$ is row affine as well.
Therefore, 
\begin{align*}
 \sum_{i=0}^n D^m_i 
 &= \sum_{i=0}^m \sum_{j=0:m} B_j^m A(j,i) 
 = \sum_{j=0}^m B_j^m \sum_{i=0:m} A(j,i) 
 = \sum_{j=0}^m B_j^m  = 1.
\end{align*}
And so $D^m$ is an affine basis:
\end{proof}

The next property is a generalization of the following property of
Bernstein polynomials:
$$
x = B^m(x)v^m = \sum_{i=0}^m \xi_i^m \, B_i^m (x)
$$
for all $x$, with $\xi^m_i := \frac{i}{m}$.

\begin{theorem}
\label{thm3}
Let $D^m = B^m E(s,:)^{-1}$ be the dual basis for some selection map $s$.
\begin{enumerate}
\item $E\xi^m = \xi^n$
\item $\xi^m = E(s,:)^{-1} \xi^n(s)$.
\item $x = B^m \xi^m = D^m(x) \xi^n(s) = \sum_{i=0}^m \xi_{s(i)}^n \, D_i^m (x) $.
\end{enumerate}
\end{theorem}

\begin{proof}
Recall that the Bernstein functions on $\$^n$ can be represented
$$\lambda_{k}^n  
  = \sum_{j=0}^{k} \frac{\binom{k}{j}}{\binom{n}{j}} \frac{1}{j!} \delta_0 D^j.
$$
In particular, $\lambda_0^n x = \delta_0 x = 0$.
For $k>0$, the sum of the terms for $\lambda^n_k x$ are zero except when $j=1$.
Hence,
$$
\lambda_{k}^n  x
  = \sum_{j=0}^{n} \frac{\binom{k}{j}}{\binom{n}{j}} \frac{1}{j!} \delta_0 D^j x
  = \frac{\binom{k}{1}}{\binom{n}{1}} \frac{1}{1!} \delta_0 D^1 x
  =  \frac{k}{n}.
$$
Therefore, $\Lambda^{nT} x = \xi^n$.
From this we get,
$$
\xi^n = \Lambda^{nT} x = \Lambda^{nT} B^m(x) \xi^m  = \Lambda^{nT} B^n(x) E \xi^m  = E \xi^m.
$$
Thereofore, $\xi^n = E \xi^m$.
Hence, $\xi^n(s) = E(s,:) \xi^m$ implies $\xi_m = E(s,:)^{-1} \xi^n(s)$,
and 
$$
D^m(x) \xi^n(s) = B^m(x) E(s,:)^{-1} \xi^n(s) = B^m(x) \xi_m = x.
$$
\end{proof}

\section{Plotting Polynomials in the Dual Bases}

Since dual bases $D^m$ are bases for the subspace $Y$,
and function in $Y$ can be represented by this basis.
In particular, in the Berntstein setup above,
any polynomial $p \in S^m$ can be represented as $p= D^m\alpha$
for some $\alpha\in\RR^{n+1}$.
And so, this basis can be used in computation with functions
in this polynomial space.

Now, since $D^m = B^m A$ with $A = E(s,:)^{-1}$,
we can write $p = B^m(A\alpha)$.
Hence, one can transform the coefficients $\alpha$ by the matrix $A$,
and then use B-form techniques in computation.
In particular, to plot the curves, one can use DeCasteljau's algortihm on the control polygon
with points $(\frac{i}{m}, (A\alpha)(i))$, for $i=0:m$.

This is depicted in Figure \ref{f2}.
The control polygon for the coefficients $\alpha$ is displayed in solid broken line
and the transformed control polygon is dashed.

\begin{figure}[H]
\includegraphics[width=3in]{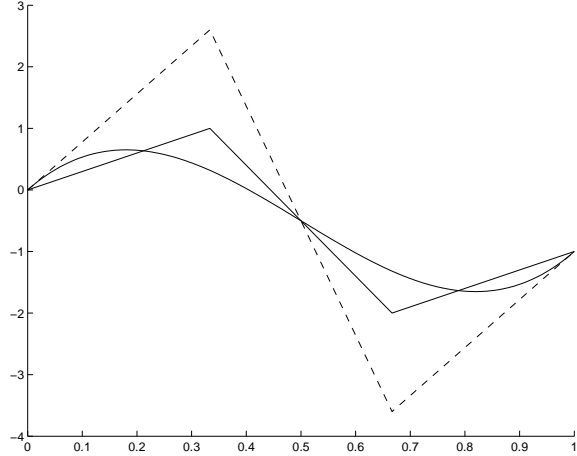}
\caption{Control polygon and transformed polygon for deCasteljau's algorithm with
dual Bernstein subspace bases.
}
\label{f2}
\end{figure}

\section{Symmetric Bernstein Class of Dual Bases}

As shown in the previous section, 
for Bernstein functionals the matrix $E(s,:) = \Lambda(s)^{nT} B^m$ is 
invertible for any selection map $s$,
and we can therefore find dual bases for any selection map.
Moreover, these bases are affine.
In this section we will use this idea to produce a certain ``symmetric'' class of bases,
and show that these converge to the Lagrange polynomial basis.
We also provide an estimate for the rate of convergence.

For $k\in\NN$, let $s(i) = ik$ for $i=0:m$. Hence, $s = 0:k:n$ with $n = k*m$.
For example:
\begin{itemize}
\item If $k=1$ then $m=m$ and we get $s = [0:m]$.
\item If $m=4$ and $k=3$, then $n = 12$ and $s = [0,3,6,9,12]$.
\end{itemize}
The dual Bernstein bases are then $D^m_k = B^m A_{m,k}$
with $A_{m,k} = E(s,:)^{-1}$.
The goal in the remainder of this section is
to show that symmetric dual bases converge to point evaluation in a certain sense.
To get the most general result, we extend the Bernstein functional to allow $k = x$ to be any real number:
$$\lambda_{x}^n := \sum_{j=0}^{\lfloor |x| \rfloor} \frac{\binom{x}{j}}{\binom{n}{j}} 
 \frac{1}{j!} \delta_0 D^j.
$$
Here, the factorials in the binomial coefficients will involve the gamma function when $x$ is a non-integer.
The functionals reduce to the above formulation when $k := x$ is a non-negative integer with $k \leq n$.
For the following, we use the \emph{falling factorial} notation
$$
(x)_j := x \cdot(x-1) \cdots (x-j+1).
$$

\begin{lemma}
\label{lem3}
For $x \in \RR$ and $j \geq 0$,
$$
\lim_{n \rightarrow \infty} \frac{\binom{xn}{j}}{\binom{n}{j}} =x^j,
$$
with $(\cdot)! := \Gamma(\cdot+1)$ for non-integer factorials.
\end{lemma}

\begin{proof}
By the well-known property $\Gamma(z+1) = z \Gamma(z)$, we get that
\begin{align*}
j!\binom{xn}{j} &= \frac{\Gamma(xn+1)}{\Gamma(xn-j+1)} 
= \frac{ (xn)_j \Gamma(xn-j+1)}{\Gamma(xn-j+1)} 
= (xn)_j.
\end{align*}
Therefore,
\begin{align*}
\lim_{n \rightarrow \infty} \frac{\binom{xn}{j}}{\binom{n}{j}} 
= \lim_{n \rightarrow \infty} \frac{xn}{n} \frac{xn-1}{n-1} \cdots \frac{xn-j+1}{n-j+1}  
= x^j.
\end{align*}
\end{proof}

\begin{prop}
Let $x \in \RR$ and $p \in \$_m$ for some $m$.
The dual functionals converge to point evaluation in the following sense
$$
\lim_{n \rightarrow \infty} \lambda_{xn}^n p = p(x).
$$
In particular, with $n = mk$ and $x = \dfrac{i}{m}$,
$$
\lim_{k \rightarrow \infty} \lambda_{ik}^{km} p = p(\frac{i}{m}).
$$
\end{prop}

\begin{proof}
Let $p(x) = \sum_{j=0}^m \alpha_j x^j$ in $\$_m$.
Then, $p^{(j)}(0) = j!\alpha_j$ if $j \leq m$ and $0$ otherwise, 
and so by the previous lemma we have
\begin{align*}
\lim_{n \rightarrow \infty} \lambda_{xn}^n p 
 = \lim_{n \rightarrow \infty} \sum_{j=0}^{ \lfloor |xn| \rfloor} 
  \frac{\binom{xn}{j}}{\binom{n}{j}} \frac{1}{j!} p^{(j)}(0) 
 = \lim_{n \rightarrow \infty} \sum_{j=0}^m \frac{\binom{xn}{j}}{\binom{n}{j}} \alpha_j
 = \sum_{j=0}^m \alpha_j \lim_{n \rightarrow \infty} \frac{\binom{xn}{j}}{\binom{n}{j}}  
 = \sum_{j=0}^{m} \alpha_j x^j 
= p(x).
\end{align*}
\end{proof}

\begin{cor}
Let $s = 0:k:km$ and $D^m_k = B^mA_k$ with $A_k = E(s,:)^{-1}$.
Then, $D^m_k \rightarrow L^m$ with $L^m$ the Lagrange basis.
\end{cor}

In figure \ref{f1}, we display the dual bases for various degree polynomial spaces
and level of refinement to illustrate convergence to the Lagrange basis.

\begin{figure}[H]
\includegraphics[width=7in]{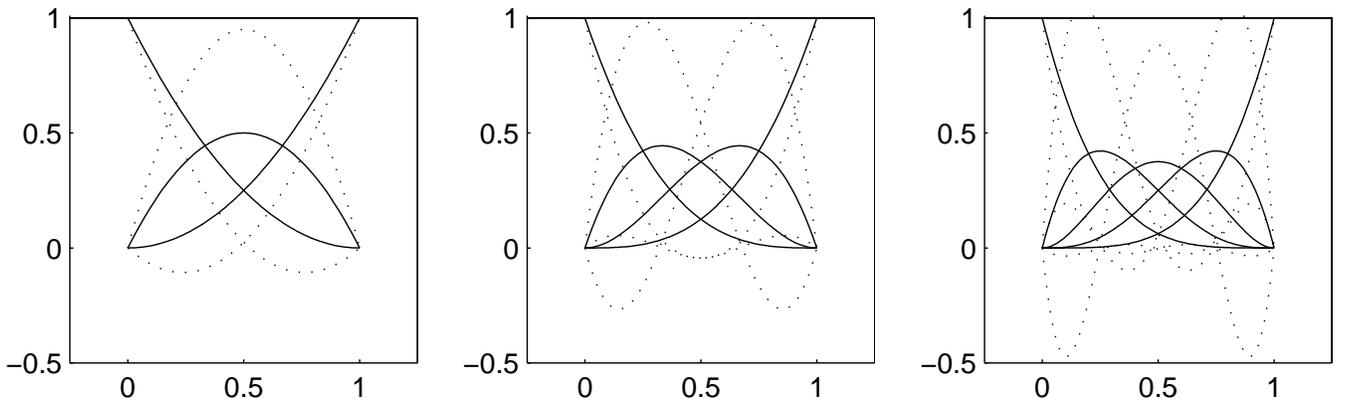}
\caption{Dual Bases $D_k^m$ of degree $m=2$, $3$ and $4$, for $k=1$ (solid)
and $10$ (dotted), respectively.
Note that $k=1$ is the Bernstein basis, and $k=10$ is close to the Lagrange basis.
}
\label{f1}
\end{figure}

\section{Rate of Convergence of Symmetric Configuration to Lagrange Interpolation}

In this section we determine the rate of convergence for this symmetric configuration to Lagrange interpolation,
which moreover provides an alternate proof or convergence to Lagrange interpolation.
Recall that $D_k^m = B^m A_k$ is the basis for $\$_m$ that is dual to $\Lambda^{mk}(s^k)$
with $s^k_i = ik$ for $i=0:m$, and $A_k = E(s^k,:)^{-1}$ with $E = (\Lambda^{mk})^T B^m$.
Let $L^m$ be the Lagrange basis for $\$_m$ dual to point evaluation at $\frac{i}{m}$ for $i=0:m$,
and let $A$ be the matrix such that $L^m = B^m A$.
Then, we have the following:

\begin{lemma}
\label{lem4}
For $h$ small:
\begin{enumerate}
\item
$(\alpha-\ell h)\cdots(\alpha-n h) 
= \alpha^{n-\ell+1} - \dfrac12 \alpha^{n-\ell}  (n+\ell)(n+1-\ell)h + O(h^2).$
\item
$\dfrac{a-bh+O(h^2)}{c-dh+O(h^2)} = \dfrac{a}{c} + \dfrac{ad-bc}{c^2} h + O(h^2)$.
\end{enumerate}
\end{lemma}

\begin{proof}
Part (2) follows by a Maclaurin's expansion.
For (1), there are $n-\ell+1$ terms. 
On expanding in powers of $h$, we have
\begin{align*}
(\alpha-\ell h) \cdots(\alpha-n h)  
&= \alpha^{n-\ell+1} -\alpha^{n-\ell} (\ell + \cdots + n) h + O(h^2) \\
&= \alpha^{n-\ell+1} -\alpha^{n-\ell}  \Big[ \binom{n+1}{2} - \binom{\ell}{2}\Big]h + O(h^2) \\
&= \alpha^{n-\ell+1} -\alpha^{n-\ell}   \frac{(n+1)n - \ell(\ell-1)}{2} h + O(h^2) \\
&= \alpha^{n-\ell+1} - \frac12 \alpha^{n-\ell}  (n+\ell)(n+1-\ell)h + O(h^2). 
\end{align*}
\end{proof}

\begin{lemma}
\label{lem5}
For the symmetric configuration,
$ A^{-1}(i,j)  - A_k^{-1}(i,j)  = 0$ if $i=0$ or $m$.
Otherwise,
$$ A^{-1}(i,j)  - A_k^{-1}(i,j)  = C_{ij} \frac{1}{k} + O(\frac{1}{k^2}) $$
as $k\rightarrow \infty$, 
with
$$
C_{ij} = \frac12 B_j^m\Big(\frac{i}{m}\Big)  
\Big[ \underbrace{\frac{(j-1)j(m-i)}{im}}_{\text{if} \ j>0}
    +\underbrace{\frac{(m-j)(2mj-im+i-ij)}{m(m-i)}}_{\text{if}\ j<m} \Big].
$$
\end{lemma}

\begin{proof}
The Lagrange basis is dual to the point evaluation map
$\Delta := [\delta_{\frac{i}{m}} : i=0:m]$,
giving $ I = \Delta^T L^m = \Delta^T B^m A $.
Therefore,
$$A^{-1} = \Delta^TB^m = [B^m_j(\frac{i}{m})].$$
We also have $A_k^{-1} = E_k(s^k,:)$  with $E$ the degree elevation matrix.
Hence,
\begin{align*}
A^{-1}(i,j)  - A_k^{-1}(i,j) = 
B^m_j(\frac{i}{m}) - E(ik,j) 
   &= B^m_j(\frac{i}{m})  - \dfrac{\binom{mk-m}{ik-j} \binom{m}{j}}{\binom{mk}{ik}}.
\end{align*}
It is easily checked that this vanishes for $i=0$ or $i=m$ (provided $k>0)$.
Hence, we only need to consider $0<i<m$.
In this case, $ik>j$ and $mk-m \geq ik-j$ for $k \geq m$,
which, since $k \rightarrow \infty$, we can assume as well.
Factoring out $\binom{m}{j}$, we can rewrite the second term as
\begin{align*}
\dfrac{1}{\binom{m}{j}} E(ik,j) 
& = 
\dfrac{\binom{mk-m}{ik-j} }{\binom{mk}{ik}}  
  = \dfrac{(ik)!}{(ik-j)!} \dfrac{(mk-m)!}{(mk)!} \dfrac{(mk-ik)!}{(mk-ik-m+j)!} \\
  &= \dfrac{(ik) \cdots (ik-j+1)}{(mk) \cdots (mk-j+1)} 
   \cdot \dfrac{(mk-ik) \cdots (mk-ik-m+j+1)}{(mk-j) \cdots (mk-m+1)} \\
  &= 
  \underbrace{\dfrac{i \cdots (i-\frac{j-1}{k})}{m \cdots (m-\frac{j-1}{k})} }_{\text{appears\ if}\ j>0}
   \cdot 
  \underbrace{\dfrac{(m-i) \cdots (m-i-\frac{m-j-1}{k})}
               {(m-\frac{j}{k}) \cdots (m-\frac{m-1}{k})} }_{\text{appears\ if}\ j<m}.
\end{align*}
By Lemma \ref{lem4}, when $j>0$, we have
\begin{align*}
 \dfrac{i \cdots (i-\frac{j-1}{k})}{m\cdots (m-\frac{j-1}{k})} 
&= \dfrac{i^j - \frac12 i^{j-1} (j-1)j \frac{1}{k} +O(\frac{1}{k^2})} 
   {m^j - \frac12 m^{j-1} (j-1)j \frac{1}{k} +O(\frac{1}{k^2})} \\
&= \dfrac{i^j}{m^j}+\dfrac{\frac12 i^jm^{j-1}(j-1)j-\frac12 m^ji^{j-1} 
   (j-1)j}{(m^j)^2} \dfrac{1}{k}+O(\frac{1}{k^2}) \\
&= \Big(\dfrac{i}{m}\Big)^j
  - \dfrac12 \Big(\dfrac{i}{m}\Big)^j \frac{(j-1)j(m-i)}{im}  
   \dfrac{1}{k}+O(\frac{1}{k^2}) \\
&= \Big(\dfrac{i}{m}\Big)^j \Big[1-\dfrac12\frac{(j-1)j(m-i)}{im}\dfrac{1}{k}\Big]
  +O(\frac{1}{k^2}).
\end{align*}
Likewise, when $j<m$,
\begin{align*}
 &\dfrac{(m-i) \cdots (m-i-\frac{m-j-1}{k})} {(m-\frac{j}{k}) \cdots (m-\frac{m-1}{k})} 
= \dfrac{ (m-i)^{m-j} - \frac12 (m-i)^{m-j-1}  (m-j-1)(m-j) \frac{1}{k} + O(\frac{1}{k^2}) }
   {m^{m-j} - \frac12 m^{m-j-1} (m+j-1)(m-j) \frac{1}{k} + O(\frac{1}{k^2})} \\
&= \Big(\dfrac{m-i}{m}\Big)^{m-j} + \dfrac{\frac12(m-i)^{m-j}m^{m-j-1} (m+j-1)(m-j)
     - \frac12 m^{m-j} (m-i)^{m-j-1}(m-j-1)(m-j)}{m^{2(m-j)}} \frac{1}{k} + O(\frac{1}{k^2}) \\
&= \Big(1-\frac{i}{m}\Big)^{m-j} + 
 \frac12\Big[\big(\frac{m-i}{m}\big)^{m-j} \frac{(m+j-1)(m-j)}{m} - 
 \big(\frac{m-i}{m}\big)^{m-j} \frac{(m-j)(m-j-1)}{m-i} \Big] \frac{1}{k} + O(\frac{1}{k^2}) \\
&= \Big(1-\frac{i}{m}\Big)^{m-j} + \frac12\big(\frac{m-i}{m}\big)^{m-j} (m-j)
 \Big[ \frac{m+j-1}{m} - \frac{m-j-1}{m-i} \Big] \frac{1}{k} + O(\frac{1}{k^2}) \\
&= \Big(1-\frac{i}{m}\Big)^{m-j} \Big[1 - \frac12 
 \frac{(m-j)(2mj-im+i-ij)}{m(m-i)} \frac{1}{k}\Big] + O(\frac{1}{k^2}).
\end{align*}
Multiplying these two terms together and by $\binom{m}{j}$ gives us, when $0<j<m$,
\begin{align*}
E(ik,j) &= 
\binom{m}{j} \dfrac{i \cdots (i-\frac{j-1}{k})}{m\cdots (m-\frac{j-1}{k})} 
  \cdot
 \dfrac{(m-i) \cdots (m-i-\frac{m-j-1}{k})} {(m-\frac{j}{k}) \cdots (m-\frac{m-1}{k})}  \\
&=
B_j^m\Big(\frac{i}{m}\Big) 
\Big[1-\dfrac12\frac{(j-1)j(m-i)}{im}\dfrac{1}{k}\Big]
\Big[1 - \frac12 \frac{(m-j)(2mj-im+i-ij)}{m(m-i)} \frac{1}{k}\Big] + O(\frac{1}{k^2}) \\
&= B_j^m\Big(\frac{i}{m}\Big) - \frac12 B_j^m\Big(\frac{i}{m}\Big) \Big[ \frac{(j-1)j(m-i)}{im}
    +\frac{(m-j)(2mj-im+i-ij)}{m(m-i)} \Big] \frac{1}{k} + O(\frac{1}{k^2}). 
\end{align*}
For $j=0$,
\begin{align*}
E(ik,0) &= B_0^m\Big(\frac{i}{m}\Big) + \frac12 B_0^m\Big(\frac{i}{m}\Big) 
     \frac{(m-1)i}{m-i} \frac{1}{k} + O(\frac{1}{k^2}).
\end{align*}
For $j=m$,
\begin{align*}
E(ik,m) &= B_m^m\Big(\frac{i}{m}\Big)  - \frac12 B_m^m \Big(\frac{i}{m}\Big)  
     \frac{(m-1)(m-i)}{i} \frac{1}{k} + O(\frac{1}{k^2}).
\end{align*}

Putting this all together, we get:
$$ A^{-1}(i,j)  - A_k^{-1}(i,j)  = C_{ij} \frac{1}{k} + O(\frac{1}{k^2})
$$
with
$$
C_{ij} = \frac12 B_j^m\Big(\frac{i}{m}\Big)  
\Big[ \underbrace{\frac{(j-1)j(m-i)}{im}}_{\text{if} \ j>0}
    +\underbrace{\frac{(m-j)(2mj-im+i-ij)}{m(m-i)}}_{\text{if}\ j<m} \Big].
$$
\end{proof}

\begin{theorem}
Let $p=L^m\alpha$ and $p_k = D^{m,k}\alpha$
with $L^m = B^mA$ the Lagrange basis
and $D^{m,k} = B^mA_k$ with $A_k = E(s_k,:)^{-1}$ the dual basis 
for the symmetric configuration given in the previous section.
Then,
$$
\norm{p - p_k}_{[0,1]} = 
  \Big(\norm{A}^2_\infty \norm{C}_\infty \frac{1}{k} + O(\frac{1}{k^2}) \Big) \norm{\alpha},
$$
and
$$
\norm{L_i(t) - D^{m,k}_i}_{[0,1]} 
  = \norm{A}_\infty^2 \norm{C}_\infty \frac{1}{k} + O(\frac{1}{k^2}).
$$
\end{theorem}

\begin{proof}
From
$A_k-A = A(A^{-1} - A_k^{-1}) A_k$,
we get
$$
\norm{A_k-A} 
\leq
\norm{A}\,\norm{A^{-1} - A_k^{-1}}\, \norm{A_k}.
$$
Since, as we proved in the previous section, that $D^{m,k}$ converges to 
the Lagrange basis $L^m$, we know that $A_k \rightarrow A$.
Hence, $\limsup_k \norm{A_k}_\infty \norm_{A}_\infty$.
We also know that 
$$\norm{A^{-1}-A_k^{-1}}_\infty = \norm{C}_\infty \frac{1}{k} + O(\frac{1}{k^2}).$$
Therefore, 
$$\norm{A_k-A}_\infty = \norm{A}_\infty^2 \norm{C}_\infty \frac{1}{k} + O(\frac{1}{k^2}.$$
Now, let $p = L^m\alpha = B^mA\alpha$ and $p_k = D^{m,k}\alpha = B^mA_k\alpha$
for $\alpha \in \RR^{m+1}$.
Then,
\begin{align*}
\norm{p - p_k}_{[0,1]} &= \norm{B^m(A-A_k)\alpha}_{[0,1]} \\
 &= \max_{t\in [0,1]} \norm{B^m(t) (A-A_k)\alpha} \\
  &\leq \max_{t\in [0,1]} \norm{B^m(t)}_\infty \norm{A-A_k}_\infty \norm{\alpha}_\infty \\
  &= \norm{A-A_k}_\infty \norm{\alpha}_\infty\\
  &= \Big(\norm{A}_\infty^2 \norm{C}_\infty \frac{1}{k} + O(\frac{1}{k^2})\Big)
   \norm{\alpha}_\infty.
\end{align*}
In particular, for $\alpha = e_i$ the standard unit vector with $1$ in the i-th slot,
we get
$$\norm{L_i(t) - D^{m,k}_i}_{[0,1]} 
  = \norm{A}_\infty^2 \norm{C}_\infty \frac{1}{k} + O(\frac{1}{k^2}).
$$
\end{proof}


\section{Basis Transformations for Symmetric Class}

Recall that for the symmetric configuration introduced in the previous section
the dual basis is represented in terms of the Bernstein basis $D^{m,k} = B^m A_{m,k}$
for some $m \times m$ matrices $A_{m,k}$.
It seems rather challenging to explicitly characterize all these transformation
matrices for arbitrary $m$ and $k$,
however, for the convenience of the reader we'll list out the first several here.
To make things more compact, we define the following notation:
$km := k-m$, $nkm := nk-m$, $pk^2nkm=pk^2-nk+m$,
and $qk^3pk^2nkm=qk^3-pk^2+nk-m$.

\vskip 10pt
\begin{tiny}
\thinmuskip=0mu \medmuskip=0mu \thickmuskip=0mu
\begin{math}
A_{2,k}= 
\dfrac{1}{2!k}
\begin{bmatrix}
   2\,k & 0 & 0\\
  -\left( k1\right)  & 2\,\left( 2k1\right)  & -\left( k1\right) \\
   0 & 0 & 2\,k
\end{bmatrix}.
\end{math}

\vskip 10pt

\begin{math}
A_{3,k} = 
\dfrac{1}{2 \cdot 3!k^2}
\begin{bmatrix}
  12\,{k}^{2} & 0 & 0 & 0\cr 
  -2\,\left( k1\right) \,\left( 5k1\right)  & 6\,\left( 2k1\right) \,\left( 3k1\right)  
      & -6\,\left( k1\right) \,\left( 3k1\right)  & 2\,\left( k1\right) \,\left( 2k1\right) \cr 
  2\,\left( k1\right) \,\left( 2k1\right)  & -6\,\left( k1\right) \,\left( 3k1\right)  
      & 6\,\left( 2k1\right) \,\left( 3k1\right)  & -2\,\left( k1\right) \,\left( 5k1\right) \cr 
  0 & 0 & 0 & 12\,{k}^{2}\end{bmatrix}.
\end{math}

\vskip 10pt

\begin{math}
A_{4,k} = 
\dfrac{1}{3 \cdot 4! k^3}
\begin{bmatrix}72\,{k}^{3} & 0 & 0 & 0&0 \cr
-3\,\left( k1\right) \,\left( 26k^29k1\right)  & 12\,\left( 2k1\right) \,\left( 3k1\right) \,\left( 4k1\right)  & -18\,\left( k1\right) \,\left( 3k1\right) \,\left( 4k1\right)  & 12\,\left( k1\right) \,\left( 2k1\right) \,\left( 4k1\right)  & -3\,\left( k1\right) \,\left( 2k1\right) \,\left( 3k1\right) \cr 
4\,\left( k1\right) \,\left( 13k^212k2\right)  & -32\,\left( k1\right) \,\left( 2k1\right) \,\left( 4k1\right)  & 24\,\left( 4k1\right) \,\left( 5k^26k2\right)  & -32\,\left( k1\right) \,\left( 2k1\right) \,\left( 4k1\right)  & 4\,\left( k1\right) \,\left( 13k^212k2\right) \cr 
-3\,\left( k1\right) \,\left( 2k1\right) \,\left( 3k1\right)  & 12\,\left( k1\right) \,\left( 2k1\right) \,\left(4k1\right)  & -18\,\left( k1\right) \,\left( 3k1\right) \,\left( 4k1\right)  & 12\,\left( 2k1\right) \,\left( 3k1\right) \,\left( 4k1\right)  & -3\,\left( k1\right) \,\left( 26k^29k1\right) \cr 
0 & 0 & 0 & 0 & 72\,{k}^{3}\end{bmatrix}.
\end{math}

\vskip 10pt

\begin{math}
\thinmuskip 0mu
\medmuskip 0mu
\thickmuskip 0mu
A_{5,k} = 
\dfrac{1}{4 \cdot 5! k^4}
\begin{bmatrix} 480\,{k}^{4} & 0 & 0 & 0 \\
  && 0 & 0\cr 
  -4\,\left( k1\right) \,\left( 7k1\right) \,\left( 22k^27k1\right)  
  & 20\,\left( 2k1\right) \,\left( 3k1\right) \,\left( 4k1\right) \,\left( 5k1\right)  
  & -40\,\left( k1\right) \,\left( 3k1\right) \,\left( 4k1\right) \,\left( 5k1\right)  
  & 40\,\left( k1\right) \,\left( 2k1\right) \,\left( 4k1\right) \,\left( 5k1\right) \cr
  && -20\,\left( k1\right) \,\left( 2k1\right) \,\left( 3k1\right) \,\left( 5k1\right)  
  & 4\,\left( k1\right) \,\left( 2k1\right) \,\left( 3k1\right) \,\left( 4\,k1\right) \cr 
2\,\left( k1\right) \,\left( 269k^3331k^2109k11\right)  
  & -10\,\left( k1\right) \,\left( 2k1\right) \,\left( 5k1\right) \,\left( 29k11\right)  
  & 20\,\left( 5k1\right) \,\left( 59k^3101k^259k11\right)  
  & -20\,\left( k1\right) \,\left( 2k1\right) \,\left( 5k1\right) \,\left( 23k11\right)  \cr
  && 10\,\left( k1\right) \,\left( 5k1\right) \,\left(37k^242k11\right)  
  & -2\,\left( k1\right) \,\left( 2k1\right) \,\left( 77k^272k11\right) \cr 
-2\,\left(k1\right) \,\left(2k1\right) \,\left( 77k^272k11\right)  
   & 10\,\left( k1\right) \,\left( 5k1\right) \,\left( 37k^242k11\right)  
  & -20\,\left( k1\right) \,\left( 2k1\right) \,\left( 5k1\right) \,\left( 23k11\right)  
  & 20\,\left( 5k1\right) \,\left( 59k^3101k^259k11\right)  \cr
  && -10\,\left( k1\right) \,\left( 2k1\right) \,\left( 5k1\right) \,\left( 29k11\right)  
  & 2\,\left( k1\right) \,\left( 269k^3331k^2109k11\right) \cr 
 4\,\left( k1\right) \,\left( 2k1\right) \,\left( 3k1\right) \,\left( 4\,k1\right)  
  & -20\,\left( k1\right) \,\left( 2k1\right) \,\left( 3k1\right) \,\left( 5\,k1\right)  
  & 40\,\left( k1\right) \,\left( 2k1\right) \,\left( 4k1\right) \,\left( 5\,k1\right)  
  & -40\,\left( k1\right) \,\left( 3k1\right) \,\left( 4k1\right) \,\left( 5\,k1\right)  \cr
  && 20\,\left( 2k1\right) \,\left( 3k1\right) \,\left( 4k1\right) \,\left( 5\,k1\right)  
  & -4\,\left( k1\right) \,\left( 7k1\right) \,\left(22k^27k1\right) \cr 
  0 & 0 & 0 & 0 \cr
  && 0 & 480\,{k}^{4}\end{bmatrix}.
\end{math}

\end{tiny}

In particular, for $m=2$ and $k = 2:5$:

\begin{tiny}
$$
\frac{1}{4}\begin{bmatrix}4 & 0 & 0\cr -1 & 6 & -1\cr 0 & 0 & 4\end{bmatrix},
\frac{1}{6}\begin{bmatrix}6 & 0 & 0\cr -2 & 10 & -2\cr 0 & 0 & 6\end{bmatrix},
\frac{1}{8}\begin{bmatrix}8 & 0 & 0\cr -3 & 14 & -3\cr 0 & 0 & 8\end{bmatrix},
\frac{1}{10}\begin{bmatrix}10 & 0 & 0\cr -4 & 18 & -4\cr 0 & 0 & 10\end{bmatrix}.
$$
\end{tiny}

For $m=3$ and $k=2:5$:

\begin{tiny}
$$
\frac{1}{48}\begin{bmatrix}48 & 0 & 0 & 0\cr -18 & 90 & -30 & 6\cr 6 & -30 & 90 & -18\cr 0 & 0 & 0 & 48\end{bmatrix},
\frac{1}{108}\begin{bmatrix}108 & 0 & 0 & 0\cr -56 & 240 & -96 & 20\cr 20 & -96 & 240 & -56\cr 0 & 0 & 0 & 108\end{bmatrix},
\frac{1}{192}\begin{bmatrix}192&0 & 0 & 0\cr -114 & 462 & -198 & 42\cr 42 & -198 & 462 & -114\cr 0 & 0 & 0 & 192\end{bmatrix},
\frac{1}{300}\begin{bmatrix}300 & 0 & 0 & 0\cr -192 & 756 & -336 & 72\cr 72 & -336 & 756 & -192\cr 0 & 0 & 0 & 300\end{bmatrix}.
$$
\end{tiny}

For $m=4$ and $k=2:4$:

\begin{tiny}
$$
\frac{1}{576}
\begin{bmatrix}576 & 0 & 0 & 0 & 0\cr -261 & 1260 & -630 & 252 & -45\cr 120 & -672 & 1680 & -672 & 120\cr -45 & 252 & -630 & 1260 & -261\cr 0 & 0 & 0 & 0 & 576\end{bmatrix}, \quad
\frac{1}{1944}\begin{bmatrix}1944 & 0 & 0 & 0 & 0\cr 
-1248 & 5280 & -3168 & 1320 & -240\cr 
664 & -3520 & 7656 & -3520 & 644\cr 
-240 & 1320 & -3168 & 5280 & -1248\cr 
0 & 0 & 0 & 0 & 1944\end{bmatrix},
$$
$$
\frac{1}{4608}\begin{bmatrix}4608 & 0 & 0 & 0 & 0\cr -3429 & 13860 & -8910 & 3780 & -693\cr 1944 & -10080 & 20880 & -10080 & 1944\cr -693 & 3780 & -8910 & 13860 & -3429\cr 0 & 0 & 0 & 0 & 4608\end{bmatrix}.
$$
\end{tiny}

For $m=5$ and $k=2:3$:

\begin{tiny}
$$
\frac{1}{38880}\begin{bmatrix}38880 & 0 & 0 & 0 & 0 & 0\cr -28480 & 123200 & -98560 & 61600 & -22400 & 3520\cr 18400 & -106400 & 238000 & -162400 & 61040 & -9760\cr -9760 & 61040 & -162400 & 238000 & -106400 & 18400\cr 3520 & -22400 & 61600 & -98560 & 123200 & -28480\cr 0 & 0 & 0 & 0 & 0 & 38880\end{bmatrix},
$$
$$
\frac{1}{122880}\begin{bmatrix}122880 & 0 & 0 & 0 & 0 & 0\cr -105300 & 438900 & -376200 & 239400 & -87780 & 13860\cr 74070 & -418950 & 906300 & -646380 & 247950 & -40110\cr -40110 & 247950 & -646380 & 906300 & -418950 & 74070\cr 13860 & -87780 & 239400 & -376200 & 438900 & -105300\cr 0 & 0 & 0 & 0 & 0 & 122880\end{bmatrix}
$$
\end{tiny}

%

\section{Quasi-Interpolation}

In this section we derive some approximation results for our dual basis functions,
similar to what is done in \cite{SL2007} in the multivariate setting.
To do so, we will redefine the Bernstein basis and dual functionals over a general interval.
\begin{itemize}
\item
Let $B^n$ with 
$$B_i^n = \binom{n}{i} \Big(\frac{b-\cdot}{b-a}\Big)^{n-i} \Big(\frac{\cdot-a}{b-a}\Big)^i$$
be the Bernstein basis over $[a,b]$, with dual functionals
$$
\lambda_{k}^n  = \sum_{j=0}^{k} 
  \frac{\binom{k}{j}}{\binom{n}{j}} \frac{(b-a)^j}{j!} \delta_a D^j.
$$
Hence, $\Lambda^{nT}B^n = I$.
\item
Let $D^m$ be the basis for $\$^m$ dual to $\Lambda^n(s)$ for some selection $s$.
That is, $\Lambda^{nT}(s) D^m = I$.
Then, $D^m = B^m A$ with $A = E(s,:)^{-1}$ where $B^n=B^mE$.
\item
Let $L^m$ be the Lagrange basis at the points $\xi^m = [a+\frac{i}{m} (b-a) : i=0:m]$
with dual map $\Delta^m = [\delta_{\xi_0^m}, \ldots, \delta_{\xi_m^m}]$.
Hence, $\Delta^{mT}L^m = I$.
Then, for $B^m = L^mM_m$ with $M_m = \Delta{^mT}B^m$, is the basis transformation.
Let $\tilde\Lambda^m = \Delta^m M_m^{-T}$.
\end{itemize}

The purpose of the map $\tilde \Lambda^m$ is two-fold.
First, it is dual to $B^m$, as we show next.
Second, it does not involve any derivative evaluations,
which makes it suitable for the approximation of functions in $C[a,b]$.

\begin{lemma}
$\tilde \Lambda^n$ is dual to $B^n$.
\end{lemma}

\begin{proof}
$$\tilde \Lambda^{nT}B^n
 = \big(\Delta^{n} M^{-T}\big)^T B^n 
 = M^{-1} \Delta^{nT}  B^n 
 = M^{-1} M = I.
$$
\end{proof}

\begin{theorem} (Stability)
Let $p = D^m\alpha \in \$^m$.
Then,
$$
\dfrac{1}{\norm{M_m^{-1}}_\infty} \norm{\alpha}_\infty
\leq \norm{p}_{[a,b]} \leq \norm{A}_\infty \norm{\alpha}_\infty.
$$
\end{theorem}

\begin{proof}
For the upper bound:
\begin{align*}
\norm{p}_{[0,1]} &= 
\norm{D^m\alpha}_{[0,1]} = \norm{B^mA\alpha}_{[0,1]}
= \norm{\sum_{i=0}^m (A\alpha)_i B_i^m }_{[0,1]} \\
&\leq \norm{A\alpha}_\infty \norm{\sum_{i=0}^m B_i^m }_{[0,1]}
= \norm{A\alpha}_\infty 
\leq \norm{A}_\infty \norm{\alpha}_\infty.
\end{align*}
For the lower bound,
note that
$$
\Delta^{mT} p = \Delta^{mT} D^m\alpha = \Delta^{mT} B^m A \alpha = M_m A \alpha,
$$
and so
$$
\alpha = A^{-1} M_m^{-1} \Delta^{mT}p.
$$
Then,
\begin{align*}
\norm{\alpha}_\infty 
  &\leq \norm{A^{-1}}_\infty \norm{M_m^{-1}}_\infty \norm{\Delta^{mT}p}_\infty \\
  &\leq \norm{A^{-1}}_\infty \norm{M_m^{-1}}_\infty \norm{p}_{[0,1]}.
\end{align*}
Since $A^{-1} = E(s,:)$ is row-affine, $\norm{A^{-1}}_\infty=1$,
and so we get
$$
\dfrac{1}{\norm{M_m^{-1}}_\infty} \norm{\alpha}_\infty
\leq \norm{p}_{[0,1]}.
$$
\end{proof}

For this we recall from Proposition \ref{prop3} that
the dual basis is unaffected by which map is used.
Hence, we restate this fact in the following lemma:

\begin{lemma}
\label{lem7}
Suppose that data maps $\Lambda^n$ and $\tilde\Lambda^n$ are
both dual to the basis $B^n$ for $\$^n$.
Then, these data maps are equivalent on $\$^n$,
and dual bases $D^m$ of $\$^m$ for $m<n$ are identical for both maps.
\end{lemma}

\begin{lemma}
For $f \in C([a,b])$, 
$$|\tilde\lambda^n_j f| \leq \norm{f}_{[a,b]} \ \norm{M^{-1}}_\infty.$$
\end{lemma}

\begin{proof}
Let $C := M^{-T}$.
Then, $\tilde\Lambda^m = \Delta C$, and so
\begin{align*}
|\tilde\lambda^m_j f| &=  |\Delta C(:,j) f| 
  = |\sum_{i=0}^n C(i,j)\delta_{\frac{i}{n}}  f| 
  = |\sum_{i=0}^n C(i,j)f(\frac{i}{n}) | \\
  &\leq \norm{f}_{[a,b]} \ \sum_{i=0}^n |C(i,j)|  
  \leq \norm{f}_{[a,b]} \ \norm{M^{-T}}_1 
  = \norm{f}_{[a,b]} \ \norm{M^{-1}}_\infty.
\end{align*}
\end{proof}

For any selection $s$, let
$$
Q_s : C([a,b]) \rightarrow \RR : f \mapsto D^m \tilde \Lambda^{n}(s)^Tf.
$$
Then, we have following:

\begin{lemma}
\label{lem6}
$Q_s$ is a linear projector of $C([a,b])$ onto $\$^m$.
\end{lemma}

\begin{proof}
Linearity is immediate from 
$$Q_s(\alpha f + \beta g)
= D^m \tilde \Lambda^n(s)^T (\alpha f + \beta g)
=\alpha D^m \tilde \Lambda^n(s)^Tf + \beta D^m \tilde \Lambda^n(s)^T g
= \alpha Q_sf + \beta Q_sg,
$$
and idempotency follows from
$$
Q_s^2 = 
(D^m \tilde \Lambda^{n}(s)^T)^2
= D^m \Big(\tilde \Lambda^{n}(s)^T D^m \Big) \tilde \Lambda^{n}(s)^T
= D^m \tilde \Lambda^{n}(s)^T
=Q_s.
$$
since $\tilde \Lambda^n(s)^T D^m = I$.
Therefore, $Q_s$ is a linear projector.
\end{proof}

\begin{theorem}
Let $f \in C[a,b]$. Then,
\begin{itemize}
\item
$\norm{Q_sf}_{[a,b]} \leq \norm{A}_\infty \norm{M_m^{-1}}_\infty  \norm{f}_{[a,b]}$,
\item
$\norm{f-Q_sf}_{[a,b]} \leq (1+\norm{Q_s}) d(f,\$_m)_{[a,b]}$.
\end{itemize}
\end{theorem}

\begin{proof}
For $t \in [a,b]$,
\begin{align*}
|Q_sf(t)| &= |(D^m \tilde \Lambda^{n}(s)^Tf) (t)|
 = |(B^m A \tilde \Lambda^{n}(s)^Tf) (t)| \\
& \leq \norm{A}_\infty \norm{\Lambda^n(s)^T f}_\infty \sum_{i=0}^m B_i^m(t) 
 = \norm{A}_\infty \norm{\Lambda^n(s)^T f}_\infty  \\
&\leq \norm{A}_\infty \norm{M^{-1}}_\infty  \norm{f}_{[a,b]}.
\end{align*}
This establishes the first result.
For the second result, 
take an arbitrary $p \in \$_m$.
Then,
\begin{align*}
\norm{f-Q_sf}_{[a,b]} 
&\leq \norm{f-p}_{[a,b]} + \underbrace{\norm{p - Q_sp}_{[a,b]}}_0 + \norm{Q_s(p-f)}_{[a,b]} \\
&\leq (1 + \norm{Q_s}_{[a,b]}) \norm{f-p}_{[a,b]}  \\
&\leq (1 + \norm{Q_s}_{[a,b]}) d(f,\$_m)_{[a,b]}.
\end{align*}

\end{proof}

\section{Bernstein-like Operator}

Let
$$ D_{m}f := D^m \Delta^{nT}(s) f  = \sum_{i=0}^m f(\xi_{s(i)}^n) D_i^m$$
be a Bernstein-like operator for our Dual functions,
with $\xi_{j}^n := a + \frac{j}{n}(b-a)$.

Let
$$
\norm{f}_{k,[a,b]} := \max \{ |f^{(k)}(x)| : x\in[a,b]\},
$$
with 
$\norm{f}_{[a,b]} := \norm{f}_{0,[a,b]}$,
and
$$
\omega(f,h) := \max \{|f(x) - f(y)| : x,y\in [a,b], |x-y| \leq h\},
$$
the uniform modulus of continuity relative to the interval $[a,b]$.
Then, we have the following:

\begin{theorem}
$$
\norm{f - D_mf}_{[a,b]} \leq 
\begin{cases}
\norm{A}_\infty \omega(f,b-a), & f \in C([a,b]); \\
(b-a)\norm{A}_\infty \norm{f'}_{[a,b]}, & f \in C^1([a,b]); \\
\frac12 (b-a)^2 \norm{A}_\infty \norm{f''}_{[a,b]}, & f \in C^2([a,b]).
\end{cases}
$$
\end{theorem}

\begin{proof}
By affineness of the basis $D^m = B^mA$ and convexity of $B^m$ for all $x \in [a,b]$,
\begin{align*}
|f(x) - D_mf(x)| 
      &= |\sum_{i=0}^m (f(x) - f(\xi^n_{s(i)}) D^m_i(x)| \\
      &\leq \omega(f,b-a) \sum_{i=0}^m |D^m_i(x)| \\
      &= \omega(f,b-a) |\sum_{j=0}^m B_j^m(x) \sum_{i=0}^m A(j,i)| \\
      &\leq \omega(f,b-a) \norm{A}_\infty \sum_{j=0}^m |B_j^m(x)| \\
      &= \norm{A}_\infty \omega(f,b-a).
\end{align*}
This gives the first case.
The second case follows from the first case and the estimate
\begin{align*}
\omega(f,h) &= \max_{|x-y|\leq h}  |f(x) - f(y)|  \\
 &= \max_{|x-y|\leq h}  \frac{|f(x) - f(y)|}{|x-y|} |x-y|  \\
 &\leq \max_{|x-y|\leq h}  \frac{|f(x) - f(y)|}{|x-y|} h  \\
 &\leq h \norm{f'}_{[a,b]}.
\end{align*}
Then, for some $\eta$ between $\xi_{s(i)}^n$ and $x$,
\begin{align*}
|f(x) - D_mf(x)|
      &= |f(x) - \sum_{i=0}^m f(\xi_{s(i)}^n)D^m_i| \\
      &= |f(x) - \sum_{i=0}^m \Big[f(x)+f'(x)(\xi_{s(i)}^n-x)+\frac12 f''(\eta) (\xi_{s(i)}^n-x)^2 \Big] D^m_i| \\
      &= |f(x) - \sum_{i=0}^m f(x) D_i^m 
                  +  \sum_{i=0}^m f'(x)(\xi_{s(i)}^n-x) D_i^m 
                  + \frac12 \sum_{i=0}^m f''(\eta) (\xi_{s(i)}^n-x)^2 D^m_i | \\
      &= |\sum_{i=0}^m f'(x)(\xi_{s(i)}^n-x) D_i^m 
                  + \frac12 \sum_{i=0}^m f''(\eta) (\xi_{s(i)}^n-x)^2 D^m_i | \\
      &= |f'(x) \sum_{i=0}^m \xi_{s(i)}^nD_i^m - f'(x) x 
                  + \frac12 \sum_{i=0}^m f''(\eta) (\xi_{s(i)}^n-x)^2 D^m_i | \\
\end{align*}
By Theorem \ref{thm3} (tranformed to the interval $[a,b]$), $\sum_{i=0}^m \xi_{s(i)}^nD_i^m =  x $,
and so
\begin{align*}
|f(x) - D_mf(x)|
    &= |f'(x) x - f'(x) x + \frac12 \sum_{i=0}^m f''(\eta) (\xi_{s(i)}^n-x)^2 D^m_i | \\
    &= |\frac12 \sum_{i=0}^m f''(\eta) (\xi_{s(i)}^n-x)^2 D^m_i | \\
   &\leq \frac12 \norm{f''}_{[a,b]} (b-a)^2 \sum_{i=0}^m |D^m_i(x)| \\
   &\leq \frac12 \norm{f''}_{[a,b]} (b-a)^2 \norm{A}.
\end{align*}
\end{proof}

\bibliographystyle{amsplain}

\end{document}